\RequirePackage{fix-cm}
\documentclass[smallcondensed]{svjour3}     
\smartqed  
\usepackage{graphicx}
\usepackage{amsmath,mathtools}
\begin{document}

\title{Affine and Conformal Submersions with Horizontal Distribution and Statistical Manifolds
}

\titlerunning{Submersions and Statistical Manifolds}        

\author{Mahesh T V \and K S Subrahamanian Moosath}

\institute{Department of Mathematics\\ Indian Institute of Space Science and Technology\\ Thiruvanathapuram, Kerala, INDIA-695547\\ 
\email{maheshtv.16@res.iist.ac.in}\\
\email{smoosath@iist.ac.in}
}

\date{Received: date / Accepted: date}

\maketitle
\begin{abstract}
We show that, for an affine submersion $\pi: \mathbf{M}\longrightarrow \mathbf{B}$ with horizontal distribution, $\mathbf{B}$ is a statistical manifold with the metric and connection induced from the statistical manifold $\mathbf{M}$. The concept of  conformal submersion with horizontal distribution is introduced, which is a generalization of affine submersion with horizontal distribution. Then proved a necessary and sufficient condition for $(\mathbf{M}, \nabla, g_M)$ to become a statistical manifold for a conformal submersion with horizontal distribution. A necessary and sufficient condition is obtained for the curve $\pi \circ \sigma$ to be a geodesic of $\mathbf{B}$, if $\sigma$ is a geodesic of $\mathbf{M}$ for $\pi: (\mathbf{M},\nabla) \longrightarrow (\mathbf{B},\nabla^*) $  a conformal submersion with horizontal distribution. Also, we obtained a necessary and sufficient condition for the tangent bundle $T\mathbf{M}$ to become a statistical manifold with respect to the Sasaki lift metric and the complete lift connection.  
\end{abstract}


%

\section{Introduction}
Information geometry has emerged from the studies of invariant geometric structures involved in statistical inference. It defines a Riemannian metric together with dually coupled affine connections in a manifold of probability distributions \cite{amari2007methods,amari2016information}. Statistical manifold was originally introduced by S.L Lauritzen\cite{lauritzen1987statistical}, later Kurose \cite{kurose1990dual} reformulated this from the view point of affine differential geometry. O'Neill \cite{o1966fundamental}  defined a Riemannian submersion and obtained fundamental equations of Riemannian submersions. Also in \cite{n1983semi} O'Neill defined a semi-Riemannian submersion. Abe and Hasegawa \cite{abe2001affine} defined affine submersions with horizontal distribution and obtained the fundamental equations. Riemannian submersions from a statistical view point was first mentioned by Barndroff-Neilsen and Jupp \cite{barndorff1988differential}. For the semi-Riemannian submersion $\pi : (M, g_M)\rightarrow (B, g_B)$ Abe and Hasegawa \cite{abe2001affine} obtained a  necessary and sufficient condition for $(M, \nabla, g_M)$ to become a statistical manifold with respect to the affine submersion with horizontal distribution $\pi : (M, \nabla)\rightarrow (B, \nabla^*)$. Conformal submersion and the fundamental equations of conformal submersion were also studied by many researchers, see \cite{lornia1993fe}, \cite{gud1990har} for example. In this paper, first we study the statistical manifold structure for affine and conformal submersions with horizontal distribution. Then compared the geodesics of $\mathbf{M}$ and $\mathbf{B}$ for conformal submersion with horizontal distribution. Also, obtained the statistical manifold structure on the tangent bundle for an affine submersion with horizontal distribution.

The projection from the tangent bundle $T\mathbf{M}$ to the manifold $\mathbf{M}$ can be considered as a submersion. Mastuzoe and Inoguchi \cite{matsuzoe2003statistical} obtained a necessary and sufficient condition for the tangent bundle $T\mathbf{M}$ to become a statistical manifold with respect to Sasaki lift metric and horizontal lift connection and also with respect to horizontal lift metric and horizontal lift connection. We have shown that the submersion $\pi : (T\mathbf{M}, \nabla^c) \rightarrow (\mathbf{M}, \nabla )$ is an affine submersion with horizontal distribution and $\pi : (T\mathbf{M}, g^{s}) \rightarrow (\mathbf{M}, g)$ is a semi-Riemannian submersion. Also obtained a necessary and sufficient condition for 
$T\mathbf{M}$ to become a statistical manifold with respect to Sasaki lift metric and complete lift connection for an affine submersion with horizontal distribution.

In section $2$, we obtained a statistical structure on the manifold $\mathbf{B}$ induced by the affine submersion $\pi: \mathbf{M}\longrightarrow \mathbf{B}$ with the horizontal distribution  $\mathcal{H}(\mathbf{M}) = \mathcal{V}^{\perp}(\mathbf{M})$.  In section $3$, we  introduced the concept of conformal submersion with horizontal distribution which is a generalization of the affine submersion with horizontal distribution. For a conformal submersion of semi-Riemannian manifolds $\pi : (\mathbf{M}, g_M)\rightarrow (\mathbf{B}, g_B)$ we proved $\pi: (\mathbf{M},\nabla) \longrightarrow (\mathbf{B}, \nabla^{*}) $ is a conformal submersion with horizontal distribution if and only if  $\pi: (\mathbf{M},\overline{\nabla}) \longrightarrow (\mathbf{B}, \overline{\nabla^{*}}) $ is a conformal submersion with the same  horizontal distribution. Then we proved a necessary and sufficient condition for $(\mathbf{M}, \nabla, g_M)$ to become a statistical manifold for conformal submersion with horizontal distribution. This is a generalization of the theorem for affine submerions with horizontal distribution proved by Abe and Hasegawa \cite{abe2001affine}. Also, for a conformal submersion with horizontal distribution we proved a necessary and sufficient condition for $\pi \circ \sigma$ to become a geodesic of $\mathbf{B}$, if $\sigma$ is a geodesic of $\mathbf{M}$.  In section $4$, first we proved that the submersion $\pi : (T\mathbf{M}, \nabla^c) \rightarrow (\mathbf{M}, \nabla )$ is an affine submersion with horizontal distribution and the submersion $\pi : (T\mathbf{M}, g^{s}) \rightarrow (\mathbf{M}, g)$ is a semi-Riemannian submersion. Finally, we obtained a necessary and sufficient condition for $(T\mathbf{M},\nabla^{c},g^{s})$ to become a statistical manifold, where $g^{s}$ is the Sasaki lift metric and $\nabla^{c}$ is the complete lift of affine connection $\nabla$ on $\mathbf{M}$.\\ Throughout this paper, all objects are assumed to be smooth. 

\section{Statistical manifolds and semi-Riemannian submersions}

Abe and Hasegawa \cite{abe2001affine} defined an affine submersion with horizontal distribution and obtained the fundamental equations. Also for the semi-Riemannian submersion $\pi : (M, g_M)\rightarrow (B, g_B)$ they obtained a  necessary and sufficient condition for $(M, \nabla, g_M)$ to become a statistical manifold with respect to the affine submersion with horizontal distribution $\pi : (M, \nabla)\rightarrow (B, \nabla^*)$. In this section, we  show that on $\mathbf{B}$ the geometric structure $(\nabla^{'},\tilde{g})$ induced by the affine submersion $\pi: \mathbf{M}\longrightarrow \mathbf{B}$ with the horizontal distribution  $\mathcal{H}(\mathbf{M}) = \mathcal{V}^{\perp}(\mathbf{M})$ is a statistical manifold structure.

A semi-Riemannian manifold $(\mathbf{M},g)$ with a torsion free affine connection $\nabla$ is called a statistical manifold if $\nabla g$ is symmetric. For a statistical manifold $(\mathbf{M},\nabla,g)$ the dual connection  $\overline{\nabla}$ is defined by  
  \begin{equation}
   Xg(Y,Z) = g(\nabla_{X}Y,Z) + g(Y,\overline{\nabla}_{X} Z)
  \end{equation}
 for  vector fields $X,Y$ and $Z$ in $\mathcal{X}(\mathbf{M})$, where $\mathcal{X}(\mathbf{M})$ denotes the set of all vector fields on $\mathbf{M}$. If $(\nabla,g)$ is a statistical structure on 
 $\mathbf{M}$ so is $(\overline{\nabla},g)$. Then $(\mathbf{M},\overline{\nabla},g)$ becomes a statistical manifold,  called the dual statistical manifold of $(\mathbf{M},\nabla,g)$.  Let $R^{\nabla}$ and $R^{\overline{\nabla}}$ be the curvature tensors of $\nabla$ and $\overline{\nabla}$, respectively. It follows from $(1)$ that 
 \begin{equation}
 g(R^{\nabla}(X,Y)Z,W) = -g(Z,R^{\overline{\nabla}}(X,Y)W)
 \end{equation}
for  $X,Y,Z$ and $W$ in  $\mathcal{X}(\mathbf{M})$. We say $(\mathbf{M},\nabla,\overline{\nabla},g)$ has constant curvature $k$ if
 \begin{equation}
 R^{\nabla}(X,Y)Z  = k \lbrace g(Y,Z)X - g(X,Z)Y \rbrace.
 \end{equation}
 A statistical manifold with  curvature zero is called a flat statistical manifold and in that case $(\mathbf{M},\nabla,\overline{\nabla},g)$  is called a dually flat statistical manifold.


Let $\mathbf{M}$ be an $n$ dimensional manifold and $\mathbf{B}$ be an $m$ dimensional manifold $(n>m)$. An onto map $\pi :\mathbf{M}\longrightarrow \mathbf{B}$ is called a submersion if $\pi_{*p} : T_{p}\mathbf{M} \longrightarrow T_{\pi(p)}\mathbf{B}$ is onto for all $p \in \mathbf{M}$. For a submersion $\pi :\mathbf{M}\longrightarrow \mathbf{B}$, $\pi^{-1}(b)$ is a submanifold of $\mathbf{M}$ of dimension $n-m$ for each $b \in \mathbf{B} $. These submanifolds $\pi^{-1}(b)$  are called fibers. Set $\mathcal{V}(\mathbf{M})_{p}$ = $Ker(\pi_{*p})$ for each $p \in \mathbf{M}$.

\begin{definition}
A submersion $\pi :\mathbf{M}\longrightarrow \mathbf{B}$ is called a submersion with horizontal distribution if there is a smooth distribution $p\longrightarrow \mathcal{H}(M)_{p}$ such that 
\begin{equation}
T_{p}\mathbf{M} = \mathcal{V}(\mathbf{M})_{p} \bigoplus \mathcal{H}(M)_{p}.
\end{equation}
\end{definition}

We call $\mathcal{V}(\mathbf{M})_{p}$ ( $\mathcal{H}(M)_{p}$) the vertical ( horizontal) subspace of $T_{p}\mathbf{M}$. $ \mathcal{H}$ and $\mathcal{V}$ denote the projection of the tangent space of $\mathbf{M}$ onto the horizontal and vertical subspaces, respectively.
\begin{note}
Let $\pi :\mathbf{M}\longrightarrow \mathbf{B}$  be a submersion with horizontal distribution  $\mathcal{H}(M)$. Then $\pi_{*}\mid_{\mathcal{H}(M)_{p}}:\mathcal{H}(M)_{p} \longrightarrow T_{\pi(p)}\mathbf{B}$ is an isomorphism for each $p \in \mathbf{M}$.
\end{note}

\begin{definition}
Let $(\mathbf{M},g_{M})$, $(\mathbf{B},g_{B})$ be semi-Riemannian manifolds of dimension  $n,m$ respectively $(n>m)$ . A submersion $\pi :\mathbf{M}\longrightarrow \mathbf{B}$ is called a semi-Riemannian submersion if all the fibers are semi-Riemannian submanifolds of $\mathbf{M}$ and $\pi_{*}$ preserves the length of horizontal vectors.
\end{definition}

\begin{note}
A vector field $Y$ on $\mathbf{M}$ is said to be projectable if there exist a vector field $Y_{*}$ on $\mathbf{B}$ such that $\pi_{*}(Y_{p})$ = $Y_{*\pi(p)}$ for each $p \in \mathbf{M}$, that is $Y$ and $Y_{*}$ are $\pi$- related. A vector field $X$ on $\mathbf{M}$ is said to be basic if it is projectable and horizontal. Every vector field $X$ on $\mathbf{B}$ has a unique smooth horizontal lift,  denoted by $\tilde{X}$, to $\mathbf{M}$. 
\end{note}

\begin{definition}
Let $\nabla$ and $\nabla^{*}$ be affine connections on $\mathbf{M}$ and $\mathbf{B}$ respectively. $\pi : (\mathbf{M},\nabla) \longrightarrow (\mathbf{B},\nabla^{*})$ is said to be an affine submersion with horizontal distribution if $\pi: \mathbf{M}\longrightarrow \mathbf{B}$ is a submersion with horizontal distribution and satisfies $\mathcal{H}(\nabla_{\tilde{X}} \tilde{Y})$ = $(\nabla^{*}_{X}Y)^{\tilde{}}$ for all vector fields $X,Y$ on $\mathbf{B}$.
\end{definition}



\begin{note} 
Abe and Hasegawa \cite{abe2001affine} proved that the connection $\nabla$ on $\mathbf{M}$ induces a connection $\nabla^{'}$ on $\mathbf{B}$ when $\pi: \mathbf{M}\longrightarrow \mathbf{B}$ is a submersion with horizontal distribution and $\mathcal{H}(\nabla_{\tilde{X}}\tilde{Y})$ is projectable for all vector fields $X$ and $Y$ on $\mathbf{B}$.\\ 

A connection $\mathcal{V}\nabla \mathcal{V}$ on the subbundle $\mathcal{V}(\mathbf{M})$ is defined by $(\mathcal{V}\nabla \mathcal{V})_{E}V$ = $\mathcal{V}(\nabla_{E}V)$ for any vertical vector field $V$ and any vector field $E$ on $\mathbf{M}$. For each $b \in \mathbf{B}$,  $\mathcal{V}\nabla \mathcal{V}$ induces a unique connection $\hat{\nabla}^{ b}$ on the fiber $\pi^{-1}(b)$. The torsion of $\nabla$ is denoted by $Tor(\nabla)$. Abe and Hasegawa \cite{abe2001affine} proved that if $\nabla$ is torsion free, then $\hat{\nabla}^{b}$ and $\nabla^{'}$ are also torsion free.
\end{note}

\begin{definition}
Let $\pi: (\mathbf{M},\nabla)\longrightarrow (\mathbf{B},\nabla^{*})$ be an affine submersion with horizontal distribution $\mathcal{V}^{\perp}(\mathbf{M})$, $g$ be a semi-Riemannian metric on $\mathbf{M}$ and $\mathcal{H}(\nabla_{\tilde{X}}\tilde{Y})$ be projectable. Define the induced semi-Riemannian metric  $\tilde{g}$ and the induced connection $\nabla^{'} $ on $\mathbf{B}$ as

\begin{eqnarray}
\tilde{g}(X,Y)&=& g(\tilde{ X},\tilde{Y})\\
 \nabla^{'}_{X}Y &=& \pi_{*}(\nabla_{\tilde{X}}\tilde{Y})
\end{eqnarray}
where $X,Y$ are vector fields on $\mathbf{B}$.
\end{definition}
Now we show that $(\mathbf{B},\nabla^{'},\tilde{g})$ is a statistical manifold.

\begin{theorem}
Let $(\mathbf{M},\nabla,g)$ be a statistical manifold and $\pi: \mathbf{M} \longrightarrow \mathbf{B}$ be an affine submersion with horizontal distribution $\mathcal{H}(\mathbf{M}) = \mathcal{V}^{\perp}(\mathbf{M})$ and $\mathcal{H}(\nabla_{\tilde{X}}\tilde{Y})$ be projectable. Then $( \mathbf{B},\nabla^{'},\tilde{g})$ is a statistical manifold.
\end{theorem}

\begin{proof}
Let $X,Y,Z$ be arbitrary vector fields on $\mathbf{B}$, we have
\begin{eqnarray*}
(\nabla^{'}_{X}\tilde{g})(Y,Z) &=& X\tilde{g}(Y,Z)-\tilde{g}(\nabla^{'}_{X}Y,Z)-\tilde{g}(Y,\nabla^{'}_{X}Z)\\
 &=& \tilde{X}g(\tilde{Y},\tilde{Z})-g(\nabla_{\tilde{X}}\tilde{Y},\tilde{Z})-g(\tilde{Y},\nabla_{\tilde{X}}\tilde{Z})\\
 &=& (\nabla_{\tilde{X}}g)(\tilde{Y},\tilde{Z})
\end{eqnarray*}
Since $(\mathbf{M},\nabla,g)$ is a statistical manifold, $( \mathbf{B},\nabla^{'},\tilde{g})$ is also a statistical manifold.
\end{proof}

\begin{definition}
 Let $\pi : (\mathbf{M},\nabla,g_{M}) \longrightarrow (\mathbf{B},\nabla^{*},g_{B})$ be an affine submersion with horizontal distribution $\mathcal{H}(M)$. Then the fundamental tensors $T$ and $A$ are defined as
\begin{eqnarray}
T_{E}F &=& \mathcal{H}(\nabla_{\mathcal{V}E}\mathcal{V}F)+\mathcal{V}(\nabla_{\mathcal{V}E}\mathcal{H}F)\\
A_{E}F &=&  \mathcal{V}(\nabla_{\mathcal{H}E}\mathcal{H}F)+\mathcal{H}(\nabla_{\mathcal{H}E}\mathcal{V}F)
\end{eqnarray}
for arbitrary vector fields $E$ and $F$ on $\mathbf{M}$. Also we denote the the fundamental tensors corresponds to the dual connection $\overline{\nabla}$ of $\nabla$ by $\overline{T}$ and $\overline{A}$.
\end{definition}

It is easy to check that these are $(1, 2)$-tensors. Note that these tensors can be defined in a general situation, namely, it is enough that a manifold $\mathbf{M}$ has a splitting $T\mathbf{M} = \mathcal{V}(\mathbf{M}) \bigoplus \mathcal{H}(\mathbf{M})$. Also note that $T_{E}$ and $A_{E}$ reverses the horizontal and vertical subspaces and $T_{E} = T_{\mathcal{V}E}$, $A_{E} = A_{\mathcal{H}E}$.

The inclusion map $(\pi^{-1}(b),\hat{\nabla}^{b}) \longrightarrow (\mathbf{M},\nabla)$ is an affine immersion in the sense of \cite{nomizu1994affine}.  The following equations are corresponding to the Gauss and Weingarten formulae.
Let $X$ and $Y$ be  horizontal vector fields, and $V$ and $W$ be vertical vector fields on $\mathbf{M}$. Then
\begin{eqnarray*}
\nabla_{V}W &=& T_{V}W+ \hat{\nabla}_{V}W\\
\nabla_{V}X &=&\mathcal{H}(\nabla_{V}X) +T_{V}X\\ 
\nabla_{X}V &=& \mathcal{V}(\nabla_{X}V) +A_{X}V\\
\nabla_{X}Y &=& \mathcal{H}(\nabla_{X}Y) +A_{X}Y
\end{eqnarray*}


\section{Conformal submersions with horizontal distribution}

Conformal submersions and the fundamental equations of conformal submersions were studied by many researchers, see \cite{lornia1993fe} and \cite{gud1990har} for example. In this section, we generalize the concept of affine submersion with horizontal distribution to conformal submersion with horizontal distribution. Then we prove $\pi: (\mathbf{M},\nabla) \longrightarrow (\mathbf{B}, \nabla^{*}) $ is a conformal submersion with horizontal distribution if and only if  $\pi: (\mathbf{M},\overline{\nabla}) \longrightarrow (\mathbf{B}, \overline{\nabla^{*}}) $ is a conformal submersion with the same  horizontal distribution. Also we obtained a necessary and sufficient condition for $(\mathbf{M}, \nabla, g_M)$ to become a statistical manifold for conformal submersion with horizontal distribution. Also obtained equations for $\mathcal{H}(E^{'})$ and $\mathcal{V}(E^{'})$ for a vector field $E$ on the curve $\sigma$ in $\mathbf{M}$ for a conformal submersion with horizontal distribution. Then for a conformal submersion with horizontal distribution we prove a necessary and sufficient condition for $\pi \circ \sigma$ to become a geodesic of $\mathbf{B}$, if $\sigma$ is a geodesic of $\mathbf{M}$.

\begin{definition}
Let $(\mathbf{M},g_{M})$ and $(\mathbf{B},g_{B})$ be Riemannian manifolds. A submersion $\pi: (\mathbf{M},g_{M}) \longrightarrow (\mathbf{B},g_{B}) $ is called a conformal submersion if there exists a $\phi \in C^{\infty}(\mathbf{M})$ such that 
\begin{equation}
g_{M}(X,Y) = e^{2 \phi}g_{B}(\pi_{*}X,\pi_{*}Y),
\end{equation}
where $X,Y \in \mathcal{X}(\mathbf{M})$.
\end{definition}

For $\pi : (\mathbf{M},\nabla) \longrightarrow (\mathbf{B},\nabla^*) $ an affine submersion with horizontal distribution $\pi_{*}(\nabla_{\tilde{X}}\tilde{Y}) = \nabla^{*}_{X}Y$, for $X,Y \in \mathcal{X}(\mathbf{B})$. In the case of conformal submersion we prove the following theorem, which is the motivation for us to generalize the concept of affine submersion with horizontal distribution. 

\begin{theorem}\label{cs12}
Let $\pi: (\mathbf{M},g_{M}) \longrightarrow (\mathbf{B},g_{B}) $  be a conformal submersion. If $\nabla$ on $\mathbf{M}$ and $\nabla^{*}$ on $B$ are the Levi-Civita connections then 
\begin{eqnarray*}
g_{B}(\pi_{*}(\nabla_{\tilde{X}}\tilde{Y}), Z) &=& g_{B}(\nabla^{*}_{X}Y, Z) - d\phi(\tilde{Z})g_{B}(X,Y)\\&+&\{d\phi(\tilde{X})g_{B}(Y,Z)+ d\phi(\tilde{Y})g_{B}(Z,X)\}
\end{eqnarray*}
where $X,Y,Z \in \mathcal{X}(\mathbf{B})$  and $\tilde{X},\tilde{Y},\tilde{Z}$ denote its unique horizontal lift on $\mathbf{M}$.
\end{theorem}

\begin{proof}
We have the Koszul formula for the Levi-Civita connection, 
\begin{eqnarray}\label{cs3}
2g_{M}(\nabla_{\tilde{X}}\tilde{Y},\tilde{Z}) &=& \tilde{X}g_{M}(\tilde{Y},\tilde{Z})+\tilde{Y}g_{M}(\tilde{Z},\tilde{X})-\tilde{Z}g_{M}(\tilde{X},\tilde{Y})-g_{M}(\tilde{X},[\tilde{Y},\tilde{Z}]) \nonumber \\ &+&g_{M}(\tilde{Y},[\tilde{Z},\tilde{X}])+g_{M}(\tilde{Z},[\tilde{X},\tilde{Y}])
\end{eqnarray}
Now consider 
\begin{eqnarray*}
\tilde{X}g_{M}(\tilde{Y},\tilde{Z}) &=& \tilde{X}(e^{2\phi}g_{B}(Y,Z))\\
                                    &=& \tilde{X}(e^{2\phi})g_{B}(Y,Z)+                          e^{2\phi}\tilde{X}(g_{B}(Y,Z))\\
                                    &=& 2e^{2\phi}d\phi(\tilde{X})g_{B}(Y,Z)+e^{2\phi}Xg_{B}(Y,Z)\\                                    
\end{eqnarray*}
Similarly we have, 
\begin{eqnarray*}
\tilde{Y}g_{M}(\tilde{X},\tilde{Z}) &=& 2e^{2\phi}d\phi(\tilde{Y})g_{B}(X,Z)+e^{2\phi}Yg_{B}(X,Z)\\
\tilde{Z}g_{M}(\tilde{X},\tilde{Y}) &=& 2e^{2\phi}d\phi(\tilde{Z})g_{B}(X,Y)+e^{2\phi}Zg_{B}(X,Y)
\end{eqnarray*}
Also we have 
\begin{eqnarray*}
g_{M}(\tilde{X}, [\tilde{Y},\tilde{Z}]) &=& e^{2\phi}g_{B}(X,[Y,Z])\\
g_{M}(\tilde{Y}, [\tilde{Z},\tilde{X}]) &=& e^{2\phi}g_{B}(Y,[Z,X])\\
g_{M}(\tilde{Z}, [\tilde{X},\tilde{Y}]) &=& e^{2\phi}g_{B}(Z,[X,Y])
\end{eqnarray*}
Then from equation (\ref{cs3}) and above equations we get 
\begin{eqnarray*}
2g_{M}(\nabla_{\tilde{X}}\tilde{Y},\tilde{Z}) &=& 2d\phi(\tilde{X})e^{2\phi}g_{B}(Y,Z) + 2d\phi(\tilde{Y})e^{2\phi}g_{B}(X,Z)\\ &-& 2d\phi(\tilde{Z})e^{2\phi}g_{B}(X,Y)+ 2e^{2\phi}g_{B}(\nabla^{*}_{X}Y,Z)
\end{eqnarray*}
This implies 
\begin{eqnarray*}
g_{B}(\pi_{*}(\nabla_{\tilde{X}}\tilde{Y}), Z) &=& g_{B}(\nabla^{*}_{X}Y, Z) - d\phi(\tilde{Z})g_{B}(X,Y)\\&+&\{d\phi(\tilde{X})g_{B}(Y,Z)+ d\phi(\tilde{Y})g_{B}(Z,X)\}
\end{eqnarray*}
\end{proof}

Now we generalize the concept of affine submersion with horizontal distribution as follows: 

\begin{definition} Let $\pi: (\mathbf{M},g_{M}) \longrightarrow (\mathbf{B},g_{B}) $  be a conformal submersion and let $\nabla$ and $\nabla^*$ be affine connections on $\mathbf{M}$ and $\mathbf{B}$, respectively. Then
$\pi : (\mathbf{M},\nabla) \longrightarrow (\mathbf{B}, \nabla^{*})$ is said to be a conformal submersion with horizontal distribution if $\pi: \mathbf{M}\longrightarrow \mathbf{B}$ is a submersion with horizontal distribution and satisfies 
 \begin{eqnarray*}
g_{B}(\pi_{*}(\nabla_{\tilde{X}}\tilde{Y}), Z) &=& g_{B}(\nabla^{*}_{X}Y, Z) - d\phi(\tilde{Z})g_{B}(X,Y)\\&+&\{d\phi(\tilde{X})g_{B}(Y,Z)+ d\phi(\tilde{Y})g_{B}(Z,X)\}
\end{eqnarray*}
for some $\phi \in C^{\infty}(\mathbf{M})$ and for all $X,Y,Z \in \mathcal{X}(\mathbf{B})$.
\end{definition}

\begin{note}
If $\phi$ is constant it turns out to be an affine submersion with horizontal distribution. 
\end{note}
\begin{example}
Let $H^{n} = \lbrace (x_{1},...,x_{n}) \in \mathbf{R}^{n} : x_{n} > 0 \rbrace$ and $\tilde{g} = \frac{1}{x_{n}^{2}} g$ be a Riemannian metric on $H^{n}$, where $g$ is the Euclidean metric on $\mathbf{R}^{n}$. Let $\pi : H^{n} \longrightarrow \mathbf{R}^{n-1}$ be the map defined by 
\begin{eqnarray*}
\pi(x_{1},...,x_{n}) = (x_{1},...,x_{n-1}).
\end{eqnarray*}
 Let $\phi : H^{n} \longrightarrow \mathbf{R}$ be a map defined by $\phi(x_{1},...,x_{n}) = \log(\frac{1}{x_{n}^{2}})$. Then we have 
\begin{eqnarray*}
\tilde{g}\left(\frac{\partial}{\partial x_{i}}, \frac{\partial}{\partial x_{j}}\right) = e^{\phi} g\left(\frac{\partial}{\partial x_{i}}, \frac{\partial}{\partial x_{j}}\right)
\end{eqnarray*}
hence, $\pi :(H^{n},\tilde{g}) \longrightarrow (\mathbf{R}^{n-1},g)$ is 
a conformal submersion. Then by the theorem (\ref{cs12}), $\pi :(H^{n},\nabla) \longrightarrow (\mathbf{R}^{n-1},\nabla^{*})$ be a conformal submersion with horizontal distribution, where $\nabla$ and $\nabla^{*}$ are Levi-Civita connection on $H^{n}$ and $\mathbf{R}^{n-1}$ respectively.
\end{example}

Now for semi-Riemannian manifolds $(\mathbf{M},g_{M})$, $(\mathbf{B},g_{B})$ with affine connections $\nabla$ and $\nabla^*$ and the dual connections $\overline{\nabla}$ and $\overline{\nabla^{*}}$ respectively we prove

\begin{proposition}
Let $\pi : (\mathbf{M},g_{M}) \longrightarrow (\mathbf{B},g_{B})$ be a conformal submersion. Then $\pi: (\mathbf{M},\nabla) \longrightarrow (\mathbf{B}, \nabla^{*}) $ is a conformal submersion with horizontal distribution $\mathcal{H}(\mathbf{M}) = \mathcal{V}(\mathbf{M})^{\perp}$ if and only if  $\pi: (\mathbf{M},\overline{\nabla}) \longrightarrow (\mathbf{B}, \overline{\nabla^{*}}) $ is a conformal submersion with the same  horizontal distribution. 
\end{proposition}

\begin{proof}
Consider, 
 \begin{eqnarray*}
 \tilde{X}g_{M}(\tilde{Y},\tilde{Z}) &=& 2e^{2\phi}d\phi(\tilde{X})g_{B}(Y,Z)+e^{2\phi}Xg_{B}(Y,Z)\\
 &=& 2e^{2\phi}d\phi(\tilde{X})g_{B}(Y,Z)+ e^{2\phi}\{g_{B}(\nabla^{*}_{X}Y,Z)+ g_{B}(Y,\overline{\nabla^{*}}_{X}Z)\}
 \end{eqnarray*}
 Now consider
 \begin{eqnarray}\label{cs4}
 \tilde{X}g_{M}(\tilde{Y},\tilde{Z}) &=& g_{M}(\nabla_{\tilde{X}}\tilde{Y}, \tilde{Z}) + g_{M}(\tilde{Y}, \overline{\nabla}_{\tilde{X}}\tilde{Z}) \nonumber\\
 &=& e^{2\phi}g_{B}(\pi_{*}(\nabla_{\tilde{X}}\tilde{Y}),Z) + e^{2\phi}g_{B}(Y,\pi_{*}(\overline{\nabla}_{\tilde{X}}\tilde{Z}))
 \end{eqnarray}
 Since, 
  \begin{eqnarray}\label{cs5}
g_{B}(\pi_{*}(\nabla_{\tilde{X}}\tilde{Y}), Z) &=& g_{B}(\nabla^{*}_{X}Y, Z) - d\phi(\tilde{Z})g_{B}(X,Y)\nonumber\\&+&\{d\phi(\tilde{X})g_{B}(Y,Z)+ d\phi(\tilde{Y})g_{B}(Z,X)\}
\end{eqnarray}
from (\ref{cs4}) and (\ref{cs5}) we get 
\begin{eqnarray*}
g_{B}(\pi_{*}(\overline{\nabla}_{\tilde{X}}\tilde{Z}), Y) &=& g_{B}(\overline{\nabla^{*}}_{X}Z, Y) - d\phi(\tilde{Y})g_{B}(X,Z)\\&+&\{d\phi(\tilde{X})g_{B}(Y,Z)+ d\phi(\tilde{Z})g_{B}(X,Y)\}
\end{eqnarray*}
Hence,  $\pi: (\mathbf{M},\overline{\nabla}) \longrightarrow (\mathbf{B}, \overline{\nabla^{*}}) $ is a conformal submersion with horizontal distribution.\\ Converse is obtained by interchanging  $\nabla$, $\nabla^{*}$  with $\overline{\nabla}$, $\overline{\nabla^{*}}$ in the above proof.
 
\end{proof}

\begin{lemma}
Let $\pi : (\mathbf{M},g_{M}) \longrightarrow (\mathbf{B},g_{B})$ be a conformal submersion and $\pi: (\mathbf{M},\nabla) \longrightarrow (\mathbf{B}, \nabla^{*}) $ be a conformal submersion with horizontal distribution $\mathcal{V}(\mathbf{M})^{\perp}$, then for horizontal vectors $X, Y$ and vertical vectors $U,V,W$ 
\begin{eqnarray}
\label{cs6}(\nabla_{\tilde{X}} g_{M})(\tilde{X}_{1},\tilde{X}_{2}) &=& e^{2\phi}(\nabla^{*}_{X}g_{B})(X_{1},X_{2})\\
\label{cs7}(\nabla_{V}g_{M})(X,Y) &=& -g_{M}(S_{V}X,Y)\\
\label{cs8}(\nabla_{X}g_{M})(V,Y) &=& -g_{M}(A_{X}V,Y)+ g_{M}(\overline{A}_{X}V,Y)\\
\label{cs9}(\nabla_{X}g_{M})(V,W) &=& -g_{M}(S_{X}V,W)\\
\label{cs10}(\nabla_{V}g_{M})(X,W) &=& -g_{M}(T_{V}X,W)+ g_{M}(\overline{T}_{V}X,W)\\
\label{cs11}(\nabla_{U}g_{M})(V,W) &=& (\hat{\nabla}_{U}\hat{g}_{M})(V,W) 
\end{eqnarray}
where $\tilde{X}_{i}$ are the horizontal lift of vector fields $X_{i}$ on $\mathbf{B}$, $\hat{g}$ is the induced metric on the fibers and $S_{V}X = \nabla_{V}X - \overline{\nabla}_{V}X$.
\end{lemma}

\begin{proof}
Consider 
\begin{eqnarray*}
(\nabla_{\tilde{X}} g_{M})(\tilde{X}_{1},\tilde{X}_{2}) &=& \tilde{X}g_{M}(\tilde{X}_{1},\tilde{X}_{2}) - g_{M}(\nabla_{X}\tilde{X}_{1},\tilde{X}_{2})- g_{M}(\tilde{X}_{1},\nabla_{X}\tilde{X}_{2})\\
&=& \tilde{X}e^{2\phi}g_{B}(X_{1},X_{2}) -e^{2\phi}g_{B}(\pi_{*}(\nabla_{\tilde{X}}\tilde{X}_{1}),X_{2})\\&-&e^{2\phi}g_{B}(X_{1},\pi_{*}(\nabla_{\tilde{X}}\tilde{X}_{2}))\\
&=& 2e^{2\phi}d\phi(\tilde{X})g_{B}(X_{1},X_{2})+ e^{2\phi}Xg_{B}(X_{1},X_{2})\\&-&e^{2\phi}g_{B}(\pi_{*}(\nabla_{\tilde{X}}\tilde{X}_{1}),X_{2})-e^{2\phi}g_{B}(X_{1},\pi_{*}(\nabla_{\tilde{X}}\tilde{X}_{2}))
\end{eqnarray*}
Since,
 \begin{eqnarray*}
g_{B}(\pi_{*}(\nabla_{\tilde{X}}\tilde{X_{i}}), X_{j}) &=& g_{B}(\nabla^{*}_{X}X_{i}, X_{j}) - d\phi(\tilde{X_{j}})g_{B}(X,X_{i})\\&+&\{d\phi(\tilde{X})g_{B}(X_{i},X_{j})+ d\phi(\tilde{X_{i}})g_{B}(X_{j},X)\}
\end{eqnarray*}
where $i,j = 1,2$ and $i \neq j$. 
We get 
\begin{eqnarray*}
(\nabla_{\tilde{X}} g_{M})(\tilde{X}_{1},\tilde{X}_{2}) &=& e^{2\phi}(\nabla^{*}_{X}g_{B})(X_{1},X_{2})
\end{eqnarray*}
Similarly we can prove other equations.
\end{proof}

Now we prove a necessary and sufficient condition for $(M, \nabla, g_M)$ to be a statistical manifold for a conformal submersion with horizontal distriution.

\begin{theorem}
 Let $\pi : (\mathbf{M},g_{M})\longrightarrow (\mathbf{B},g_{B}) $ be a conformal submersion and $\pi: (\mathbf{M},\nabla)\longrightarrow (\mathbf{B},\nabla^{*}) $ be a conformal submersion with horizontal distribution $\mathcal{H}(\mathbf{M})$\\ $= \mathcal{V}^{\perp}(\mathbf{M})$ and let $Tor(\nabla) = 0$. Then $(\mathbf{M},\nabla,g_{M})$ is a statistical manifold if and only if 
\begin{enumerate}
\item $\mathcal{H}(S_{V}X) = A_{X}V - \overline{A}_{X}V$
\item $\mathcal{V}(S_{X}V) = T_{V}X - \overline{T}_{V}X$
\item $(\pi^{-1}(b),\hat{\nabla}^{b},\hat{g}_{M}^{b})$ is a statistical manifold for each $b\in \mathbf{B}$.
\item $(\mathbf{B},\nabla^{*},g_{B})$ is a statistical manifold.
\end{enumerate}
\end{theorem}

\begin{proof}
Suppose $(\mathbf{M},\nabla,g_{M})$ is a statistical manifold, then $\nabla g_M$ is symmetric. Then we have
\begin{eqnarray*}
(\nabla_{V}g_{M})(X,Y) = (\nabla_{X}g_{M})(V,Y).
\end{eqnarray*}
Where $X,Y$ are horizontal vectors and $V$ is a vertical vector. Then from (\ref{cs7}) and (\ref{cs8}) of above lemma we get 
\begin{eqnarray*}
-g_{M}(S_{V}X,Y) &=& -g_{M}(A_{X}V,Y)+ g_{M}(\overline{A}_{X}V,Y)\\
\end{eqnarray*}
This implies, 
\begin{eqnarray*}
\mathcal{H}(S_{V}X) &=& A_{X}V - \overline{A}_{X}V.
\end{eqnarray*}

Similarly from (\ref{cs9}) and (\ref{cs10}) of above lemma we get 
\begin{eqnarray*}
\mathcal{V}(S_{X}V) = T_{V}X - \overline{T}_{V}X
\end{eqnarray*}
Since $\nabla g_M$ is symmetric from (\ref{cs11}) of above lemma we get $\hat{\nabla}^{b}\hat{g}^{b}$ is symmetric, so  $(\pi^{-1}(b),\hat{\nabla}^{b},\hat{g}_{M}^{b})$ is a statistical manifold. 
Also from (\ref{cs6}) of above lemma we get, 
\begin{eqnarray*}
(\nabla_{\tilde{X}} g_{M})(\tilde{X}_{1},\tilde{X}_{2}) &=& e^{2\phi}(\nabla^{*}_{X}g_{B})(X_{1},X_{2}).
\end{eqnarray*}
Where $\tilde{X}_{i}$ are the horizontal lift of vector fields $X_{i}$ on $\mathbf{B}$. Since, $\nabla g_{M}$ is symmetric $\nabla^{*}g_{B}$ is also symmetric. Hence, $(\mathbf{B},\nabla^{*},g_{B})$ is a statistical manifold.\\
 
  Conversely, if all the four conditions hold then from the above lemma we get 
\begin{eqnarray*}
\nabla_{E}g_{M}(F,G) = \nabla_{F}g_{M}(E,G)
\end{eqnarray*}
for any vector fields $E,F$ and $G$ on $\mathbf{M}$. That is, $\nabla g_{M}$ is symmetric on $\mathbf{M}$, so  $(\mathbf{M},\nabla,g_{M})$ is a statistical manifold.

\end{proof}

\subsection{Geodesics}

Let $M$, $B$ be Riemannian manifolds and $\pi :M\rightarrow B$ be a submersion. Let $E$ be a vector field on a curve $\sigma$ in $\mathbf{M}$ and the horizontal part $\mathcal{H}(E)$ and the vertical part $V=\mathcal{V}(E)$ of $E$ be denoted by $H$ and $V$, respectively. $\pi \circ \sigma$ is a curve in $B$ and $E_{*}$ denote the vector field $\pi_{*}(E) = \pi_{*}(H)$ on the curve $\pi \circ \sigma$ in $\mathbf{B}$. $E_{*}^{'}$ denote the covariant derivative of $E_{*}$ and is a vector field on $\pi \circ \sigma$. The horizontal lift to $\sigma$ of $E_{*}^{'}$ is denoted by $\tilde{E_{*}^{'}}$. In \cite{o1967submersions} O'Neil has shown that

$$\mathcal{H}(E^{'}) = \tilde{E_{*}^{'}} + A_H(U) + A_X(V) + T_U(V)$$

$$\mathcal{V}(E^{'}) = A_{X}H+T_{U}H+\mathcal{V}(V^{'}),$$

where $X= \mathcal{H}(\sigma^{'})$ and $U= \mathcal{V}(\sigma^{'})$ and $A$, $T$ are fundamental tensors.\\

Using this O'Neil \cite{o1967submersions} has proved that if $\sigma$ is a geodesic of $\mathbf{M}$, then $\pi \circ \sigma$ is a geodesic of $\mathbf{B}$ if and only if $2A_{X}U+T_{U}U=0$. For affine submersion $\pi : (\mathbf{M}, \nabla) \rightarrow (\mathbf{B},\nabla^{*})$ with horizontal distribution $\mathcal{H}(M)$, Abe and Hasegawa \cite{abe2001affine} proved a similar result.

We consider $\pi :( \mathbf{M}, \nabla, g_{m} )\rightarrow(\mathbf{B},\nabla^{*}, g_{b})$ a conformal submersion with horizontal distribution $\mathcal{H}(\mathbf{M})$. Throughout this section we assume $\nabla$ is torsion free. A curve $\sigma$ is a geodesic if and only if $\mathcal{H}(\sigma ^{''})=0$, where $\sigma^{''}$ is the covariant derivative of $\sigma^{'}$. So first we obtain equations for $\mathcal{H}(E^{'})$ and $\mathcal{V}(E^{'})$ for a vector field $E$ on curve $\sigma$ in $\mathbf{M}$ for a conformal submersion with horizontal distribution.  

\begin{theorem}
Let $ \pi: (\mathbf{M},\nabla,g_{m}) \rightarrow (\mathbf{B},\nabla^{*},g_{b})$ be a conformal submersion with horizontal distribution, and let $E = H + V$ be a vector field on curve $\sigma$ in $\mathbf{M}$. Then for every vector field $Z$ on $\mathbf{B}$ we have
\begin{eqnarray*}
g_{b}(\pi_{*}(\mathcal{H}(E^{'})),Z) &=& g_{b}((E_{*}^{'})+\pi_{*}(A_{X}U+A_{X}V+T_{U}V),Z)-d\phi(\tilde{Z})g_{b}(\pi_{*}X,\pi_{*}H)\\&+&d\phi(X)g_{b}(\pi_{*}H,Z)+d\phi(H)g_{b}(\pi_{*}X,Z).\\
\mathcal{V}(E^{'}) &=& A_{X}H+T_{U}H+\mathcal{V}(V^{'})
\end{eqnarray*}
where $X= \mathcal{H}(\sigma^{'})$ and $U= \mathcal{V}(\sigma^{'})$.
\end{theorem}

\begin{proof}
We work in a neighbourhood of an arbitrary point $\sigma(t)$ of the curve.  Let $B_{1},....,B_{n} (n= dim \mathbf{B})$ be a base field near $\pi(\sigma(t))$ on $\mathbf{B}$ and $\tilde{B}_{1}, ...,\tilde{B}_{n}$ their horizontal lift. Let $F_{1},...,F_{l}$ be a vertical base field near $\sigma(t)$,where $l$ is the dimension of fibers. Set $H = \sum _{i}a^{i}(\tilde{B_{i}}\mid_{\sigma})$,$V = \sum _{j}b^{j}(F_{j}\mid_{\sigma})$, $X = \sum _{k}c^{k}(\tilde{B_{k}}\mid_{\sigma})$ and $ U= \sum _{m}d^{m}(F_{m}\mid_{\sigma})$, where $a^{i},b^{j}, c^{k}$ are real valued function on a neighbourhood of $t$ and $\tilde{B_{i}}\mid_{\sigma}$ (resp. $F_{m}\mid_{\sigma}$) are vector fields along curve $\sigma$ defined by $(\tilde{B_{i}}\mid_{\sigma})_{s} := (\tilde{B_{i}})_{\sigma(s)}$ (resp.$F_{m}\mid_{\sigma}:= (F_{m})_{\sigma(s)}$). We have

 \begin{eqnarray}\label{cg1}
  \begin{split}
 \mathcal{H}(E^{'})_{t} = &\mathcal{H}\left\{ \sum_{i}a^{i'}(t)(\tilde{B_{i}})_{\sigma(t)}+ \sum_{i}a^{i}(t)(\nabla_{\sigma^{'}(t)}\tilde{B_{i}})_{\sigma(t)}\right.\\ &\left.+ \sum_{j}b^{j'}(t)(F_{j})_{\sigma(t)}+ \sum_{j}b^{j}(t)(\nabla_{\sigma^{'}(t)}F_{j})_{\sigma(t)} \right\} \\
 =&\sum_{i}a^{i'}(t)(\tilde{B_{i}})_{\sigma(t)} + \sum_{i,k}a^{i}(t)c^{k}(t)\mathcal{H}(\nabla_{\tilde{B_{k}}}\tilde{B_{i}})_{\sigma(t)}\\
 &+ \sum_{i,m} a^{i}(t)d^{m}(t)\mathcal{H}(\nabla_{F_{m}}\tilde{B_{i}})_{\sigma(t)}+(A_{X}V)_{\sigma(t)}+ (T_{U}V)_{\sigma(t)}.
 \end{split}
  \end{eqnarray}
 Since $\mathcal{H}[\tilde{B_{i}},F_{m}] = 0$  and $\nabla$ is torsion free we get 
 \begin{eqnarray}\label{cg2}
 \sum_{i,m} a^{i}(t)d^{m}(t)\mathcal{H}(\nabla_{F_{m}}\tilde{B_{i}})_{\sigma(t)} = (A_{H}U)_{\sigma(t)}. 
 \end{eqnarray}
Also we have
\begin{eqnarray}
(E_{*}^{'})_{t} &=& \sum_{i} \lbrace a^{i'}(t)(B_{i})_{\pi(\sigma(t))}+a^{i}(t)(B_{i}\mid_{\pi \circ \sigma})^{'}_{\pi(\sigma(t))} \nonumber\\
&=& \sum_{i}a^{i'}(t)(B_{i})_{\pi(\sigma(t))} + \sum_{i,k} a^{i}(t)c^k(t)(\nabla^{*}_{B_{k}}B_{i})_{\pi(\sigma(t))}\label{cg3}
\end{eqnarray}   

From (\ref{cg1}) and (\ref{cg2}) we get 
\begin{eqnarray}\label{cg4}
\pi_{*}(\mathcal{H}(E^{'})_{t}) &=& \sum_{i} a^{i'}(t)(B_{i})_{\pi(\sigma(t))}+ \sum_{i,k}a^{i}(t)c^{k}(t)\pi_{*}(\mathcal{H}(\nabla_{\tilde{B_{k}}}\tilde{B_{i}}))_{\pi(\sigma(t))} \nonumber\\ 
&+& \pi_{*}((A_{H}U)+(A_{X}V)+(T_{U}V))_{\pi(\sigma(t))}
\end{eqnarray}

Now from (\ref{cg3}), (\ref{cg4}) and definition of conformal submersion with horizontal distribution we get
\begin{eqnarray*}
g_{b}(\pi_{*}(\mathcal{H}(E^{'})),Z) &=& g_{b}((E_{*}^{'})+\pi_{*}(A_{X}U+A_{X}V+T_{U}V),Z)-d\phi(\tilde{Z})g_{b}(\pi_{*}X,\pi_{*}H)\\&+&d\phi(X)g_{b}(\pi_{*}H,Z)+d\phi(H)g_{b}(\pi_{*}X,Z).
\end{eqnarray*}

Similarly we can prove $\mathcal{V}(E^{'}) = A_{X}H+T_{U}H+\mathcal{V}(V^{'}).$
\end{proof}
\vspace{.2cm}
 For $\sigma^{''}$ we have

\begin{corollary}
Let $\sigma$ be a curve in $\mathbf{M}$ with $X = \mathcal{H}(\sigma^{'})$ and $U = \mathcal{V}(\sigma^{'})$. Then 
\begin{eqnarray}
g_{b}(\pi_{*}(\mathcal{H}(\sigma^{''})),Z) &=& g_{b}((\sigma_{*}^{''})+\pi_{*}(2A_{X}U+T_{U}U),Z)-d\phi(\tilde{Z})g_{b}(\pi_{*}X,\pi_{*}X)\nonumber\\&+&2d\phi(X)g_{b}(\pi_{*}X,Z) \label{cg5}\\
\mathcal{V}(\sigma^{''}) &=& A_{X}X+T_{U}X+\mathcal{V}(U^{'}) \label{cg6}
\end{eqnarray}
where $\sigma_{*}^{''}$ denotes the covariant derivative of $(\pi \circ \sigma)^{'}$.
\end{corollary}

Now for a conformal submersion with horizontal distribution we prove a necessary and sufficient condition for $\pi \circ \sigma$ to become a geodesic of $\mathbf{B}$, if $\sigma$ is a geodesic of $\mathbf{M}$.

 \begin{theorem}
 Let $ \pi: (\mathbf{M},\nabla,g_{m}) \rightarrow (\mathbf{B},\nabla^{*},g_{b})$ be a conformal submersion with horizontal distribution. If $\sigma$ is a geodesic on $\mathbf{M}$, then $\pi \circ \sigma$ is a geodesic on $\mathbf{B}$ if and only if 
 \begin{eqnarray*}
 g_{b}(\pi_{*}(2A_{X}U+T_{U}U),Z)+2d\phi(X)g_{b}(\pi_{*}X,Z) = d\phi(\tilde{Z})\parallel \pi_{*}X \parallel^{2}
 \end{eqnarray*}
 for every vector field $Z$ on $\mathbf{B}$, where $X = \mathcal{H}(\sigma^{'})$ and $U = \mathcal{V}(\sigma^{'})$ and $\parallel \pi_{*}X \parallel^{2} = g_{b}(\pi_{*}X,\pi_{*}X) $
 \end{theorem}
 
\begin{proof}
Since $\sigma$ is a geodesic on $\mathbf{M}$ from (\ref{cg5}) we get
\begin{eqnarray*}
 g_{b}((\sigma_{*}^{''}), Z) = d\phi(\tilde{Z})\parallel \pi_{*}X \parallel^{2} - g_{b}(\pi_{*}(2A_{X}U+T_{U}U),Z) - 2d\phi(X)g_{b}(\pi_{*}X,Z)
\end{eqnarray*}
for every vector field $Z$ on $\mathbf{B}$. Hence, $\pi \circ \sigma$ is a geodesic on $\mathbf{B}$ if and only if 
 \begin{eqnarray*}
 g_{b}(\pi_{*}(2A_{X}U+T_{U}U),Z)+2d\phi(X)g_{b}(\pi_{*}X,Z) = d\phi(\tilde{Z})\parallel \pi_{*}X \parallel^{2}
 \end{eqnarray*}
 for every vector field $Z$ on $\mathbf{B}$.  
\end{proof}
\begin{remark}
If $\sigma$ is a horzontal geodesic (that is, $\mathcal{V}(\sigma^{'})=0$), then $\pi \circ \sigma$ is a geodesic. Also from (\ref{cg6}), horizontal lift of a geodesic of $\mathbf{B}$ is a geodesic of $\mathbf{M}$ if $A_{Z}Z = 0$ for all horizontal vector fields $Z$ on $\mathbf{M}$.
\end{remark}
\section{Statistical structures on the tangent bundle}

Theory of horizontal lift and complete lift of tensor fields and connections on the tangent bundle are well known\cite{yano1973tangent}. 
Mastuzoe and Inoguchi \cite{matsuzoe2003statistical} obtained necessary and sufficient condition for the tangent bundle $T\mathbf{M}$ to become a statistical manifold with respect to the Sasaki lift metric and the horizontal lift connection and also with respect to the horizontal lift metric and the horizontal lift connection.  In this section, first we show that the submersion $\pi : (T\mathbf{M}, \nabla^c) \rightarrow (\mathbf{M}, \nabla )$ is an affine submersion with horizontal distribution and $\pi : (T\mathbf{M}, g^{s}) \rightarrow (\mathbf{M}, g)$ is a semi-Riemannian submersion. Also prove a necessary and sufficient condition for $T\mathbf{M}$ to become a statistical manifold with respect to Sasaki lift metric and complete lift connection for an affine submersion with horizontal distribution.
 
Let $\mathbf{M}$ be an $n$-dimensional manifold and $T\mathbf{M} = \amalg_{x\in \mathbf{M}}T_{x}\mathbf{M}$ denote the tangent bundle on $\mathbf{M}$. Let $\pi: T\mathbf{M} \longrightarrow \mathbf{M}$ be the natural projection defined by $X_{x}\in T_{x}\mathbf{M}\longrightarrow x \in \mathbf{M}$. Let $(U; x^{1},...,x^{n})$ be a local coordinate system on $\mathbf{M}$ and the induced co-ordinate system on $\pi^{-1}(U)$ be $(x^{1},..x^{n}; u^{1},..u^{n})$. Let $(x;u)$ be a point on $T\mathbf{M}$,  denote the kernel of $\pi_{*(x;u)}$ by $\mathcal{V}_{(x;u)}$ called the vertical subspace of $T_{(x;u)}(T\mathbf{M})$ at $(x;u)$. Note that the vertical subspace $\mathcal{V}_{(x;u)}$ is spanned by $\lbrace \frac{\partial}{\partial u^{1}},\frac{\partial}{\partial u^{2}},...\frac{\partial}{\partial u^{n}}\rbrace$. The two linear spaces $T_{x}\mathbf{M}$ and $\mathcal{V}_{(x;u)}$ have same dimension, so there is a canonical linear isomorphism $ V:T_{x}\mathbf{M} \longrightarrow \mathcal{V}_{(x;u)}$ called the vertical lift.

Let $f:\mathbf{M}\longrightarrow \mathbf{R} $ be a smooth function on
 $\mathbf{M}$ and $\pi: T\mathbf{M} \longrightarrow \mathbf{M}$ be the natural projection. The vertical lift of $f$ is denoted by  $f^{v}$ and defined as $  f^{v} = f \circ \pi $. For a vector field $X = X^{i}\frac{\partial}{\partial x^{i}}$ on $\mathbf{M}$ the vertical lift is denoted by $X^{v}$ and defined as $X^{v} = (X^{i})^{v}\frac{\partial}{\partial u^{i}} $. Note that for any vector fields $X,Y$ on $\mathbf{M}$  $[X^{v},Y^{v}] = 0$. The vertical lift of $df$ is defined by $(df)^{v} = d(f^{v})$, in particular for local co-ordinate functions $x^{i}$, $(dx^{i})^{v} = d(x^{i})^{v}$. The vertical lift of $1$-form $\omega = \omega_{i}dx^{i}$ is defined as $\omega^{v} = (\omega_{i})^{v}d(x^{i})^{v}$, the vertical lift operation extends on the full tensor algebra $\mathcal{T}(\mathbf{M})$ by the rule $(P\otimes Q)^{v} = P^{v}\otimes Q^{v}$ 
for any tensor fields $P$ and $Q$ on $\mathbf{M}$.

Let $f: \mathbf{M}\longrightarrow \mathbf{R}$ be a smooth map, the complete lift $f^{c}$ of $f$ on $T\mathbf{M}$ is defined as $f^{c}= i(df) = u^{i}\frac{\partial f}{\partial x^{i}}$. The complete lift $X^{c}$  on $T\mathbf{M}$ of the vector field $X$ on $\mathbf{M}$ is characterized by the formula $X^{c}(f^{c}) = (Xf)^{c}$ for all $f \in C^{\infty}(\mathbf{M})$. In local co-ordinate, the complete lift $X^{c}$ of $X = X^{i}\frac{\partial}{\partial x^{i}}$ has the local expression
\begin{eqnarray*}
X^{c} &=& (X^{i})\frac{\partial}{\partial x^{i}}+ u^{j}\frac{\partial X^{i}}{\partial x^{j}} \frac{\partial}{\partial u^{i}}.
\end{eqnarray*}

The complete lift to $1-$from $\omega$ is defined as 
\begin{eqnarray*}
\omega^{c}(X^{c}) &=& (\omega(X))^{c}.
\end{eqnarray*}
More generally the complete lift to full tensor algebra  $\mathcal{T}(\mathbf{M})$ is given by the rule 
\begin{eqnarray*}
(P\otimes Q)^{c} = P^{c}\otimes Q^{v} + P^{v}\otimes Q^{c}
\end{eqnarray*}
for any tensor fields $P$ and $Q$ on $\mathbf{M}$. Let $\nabla$ be a linear connection on $\mathbf{M}$, then the complete lift $\nabla^{c}$ on $T\mathbf{M}$ is defined as $\nabla^{c}_{X^{c}}Y^{c} = (\nabla_{X}Y)^{c}$ for every $X,Y \in \mathcal{X}(\mathbf{M})$.

\begin{remark}
Matsuzoe and Inoguchi \cite{matsuzoe2003statistical} have proved that if $(\mathbf{M},\nabla,g)$ is a statistical manifold, then  $(T\mathbf{M},\nabla^{c},g^{c})$ is a statistical manifold. Moreover the conjugate connection of $\nabla^{c}$ is $\overline{(\nabla^{c})} = (\overline{\nabla})^{c}$.
\end{remark}

Now we look at the horizontal lifts on the tangent bundle. Let $\mathbf{M}$ be a smooth $n-$dimensional manifold and $\nabla$ be a torsion free linear connection on $\mathbf{M}$. The vertical subspace $\mathcal{V}_{(x;u)}$ of $T_{(x;u)}(T\mathbf{M})$ at $(x;u)$ defines a smooth distribution $\mathcal{V}$ on $T\mathbf{M}$ called the vertical distribution. Also there exists a smooth distribution $x \longrightarrow \mathcal{H}(T\mathbf{M})_{x}$ depending on the linear connection $\nabla$ such that 
 \begin{eqnarray*}
 T_{(x;u)}(T\mathbf{M}) = \mathcal{H}(T\mathbf{M})_{x}\oplus \mathcal{V}_{(x;u)}.
 \end{eqnarray*}
This distribution is denoted by $\mathcal{H}_{\nabla}$, called the horizontal distribution. Let $X$ be a vector field on $\mathbf{M}$, then the horizontal lift of $X$ on $T\mathbf{M}$ is the unique vector field $X^{H}$ on $T\mathbf{M}$ such that $\pi_{*}(X^{H}_{(x;u)})= X_{\pi((x;u))}$ for all $(x;u) \in T\mathbf{M}$. In local co-ordinates if $X = X^{i}\frac{\partial}{\partial x^{i}}$, then 
 \begin{eqnarray*}
 X^{H} = X^{i}\frac{\partial}{\partial x^{i}} - X^{j}u^{k}\Gamma_{j,k}^{i}\frac{\partial}{\partial u_{i}}.
 \end{eqnarray*}
Here $\Gamma_{j,k}^{i}$is the connection coefficient of $\nabla$.\par
 Let $g$ be a semi-Riemannian metric on $\mathbf{M}$, then the horizontal lift $g^H$ on $\mathbf{M}$ is defined as  $g^{H}(X^{H}, Y^{H})= g^{H}(X^{v},Y^{v}) = 0$ and $g^{H}(X^{H},Y^{v}) = g(X,Y)$, for $X,Y \in \mathcal{X}(\mathbf{M})$. The horizontal lift $\nabla^H$ on $\mathbf{M}$ of linear connection $\nabla$ on $\mathbf{M}$ is defined as $\nabla^{H}_{X^{v}}Y^{v} = 0$, $\nabla^{H}_{X^{v}}Y^{H} = 0$, $\nabla^{H}_{X^{H}}Y^{v} = (\nabla _{X} Y)^{v}$, $\nabla^{H}_{X^{H}}Y^{H} = (\nabla _{X}Y)^{H}$, for $X,Y \in \mathcal{X}(\mathbf{M})$. Note that even if $\nabla$ is torsion free, its horizontal lift $\nabla^{H}$ may have non-trivial torsion.

Let $g$ be a semi-Riemannian metric on $(\mathbf{M},\nabla)$. We define a semi-Riemannian  metric $g^{s}$ on $T\mathbf{M}$ as,  $g^{s}_{(x;u)}(X^{H},Y^{H}) = g_{x}(X,Y)$, $g^{s}_{(x;u)}(X^{H},Y^{v}) = 0$,\\ $g^{s}_{(x;u)}(X^{v},Y^{v}) = g_{x}(X,Y)$. The metric $g^{s}$ is called the Sasaki lift metric.

In \cite{yano1973tangent} Yano and Ishihara introduced $\gamma$ operator for defining horizontal lift from complete lift. Let $X$ be a vector field on $\mathbf{M}$, with local expression $X = X^{i}\frac{\partial}{\partial x^{i}}$, $\nabla X = X_{j}^{i}\frac{\partial}{\partial x_{i}}\otimes dx^{j}$, where $X_{j}^{i}= \frac{\partial X^{i}}{\partial x^{j}}+ X^{k}\Gamma_{j,k}^{i}$.  Define $\gamma(\nabla X) = u^{j}X^{i}_{j}\frac{\partial}{\partial u_{i}}$  with respect to the induced co-ordinate $(x^{1},...,x^{n}; u^{1},...,u^{n})$. Then we can see that $X^{H} = X^{c} - \gamma(\nabla X)$, note that $\gamma(\nabla X)$ is the vertical part of $X^{c}$. 

\begin{remark}
Matsuzoe and Inoguchi \cite{matsuzoe2003statistical} proved that if $(\mathbf{M},\nabla, g)$ is a statistical manifold, then $(T\mathbf{M},  \nabla^{H},g^{s})$ or $(T\mathbf{M},\nabla^{H},g^{H})$ is a statistical manifold if and only if $\nabla g = 0$. Also they obtained that for a statistical manifold $(\mathbf{M},\nabla,g)$ both $(T\mathbf{M}, g^{s},C^{H})$ and $(T\mathbf{M}, g^{H},C^{H})$ are  statistical manifolds, where $C^{H}$ is the horizontal lift of the cubic form $C = \nabla g$.
\end{remark}


Consider the submersion $\pi: T\mathbf{M}\longrightarrow \mathbf{M}$. Let $\nabla$ be an affine connection on $\mathbf{M}$. Then there is a horizontal distribution $\mathcal{H}_{\nabla}$ such that 
\begin{eqnarray*}
T_{(x;u)}(T\mathbf{M}) = \mathcal{H}_{(x;u)}(T\mathbf{M}) + \mathcal{V}_{u}
\end{eqnarray*}
for every $(x;u)\in T\mathbf{M}$.\par Now we show that the submersion $\pi$ of $T\mathbf{M}$  into $\mathbf{M}$ with complete lift of affine connection is an affine submersion with horizontal distribution.
 
\begin{proposition}\label{csp1}
The submersion $\pi : (T\mathbf{M}, \nabla^{c}) \longrightarrow (\mathbf{M},\nabla)$ is an affine submersion with horizontal distribution.
\end{proposition}

\begin{proof}
We need to show that 
\begin{eqnarray*}
\mathcal{H}(\nabla^{c}_{X^{H}}Y^{H}) = (\nabla_{X}Y)^{H}
\end{eqnarray*}
Consider, $X^{H}= X^{c}-\gamma(\nabla X)$, then  
\begin{eqnarray*}
\nabla^{c}_{X^{H}}Y^{H} &=& \nabla^{c}_{X^{c}-\gamma(\nabla X)} Y^{c}-\gamma(\nabla Y)\\
&=& \nabla^{c}_{X^{c}-\gamma(\nabla X)} Y^{c} -\nabla^{c}_{X^{c}-\gamma(\nabla X)} \gamma(\nabla Y)\\
&=& \nabla^{c}_{X^{c}}Y^{c} - \nabla^{c}_{\gamma(\nabla X)}Y^{c}-  \nabla^{c}_{X^{c}}\gamma(\nabla Y)+\nabla^{c}_{\gamma(\nabla X)}\gamma(\nabla Y)
\end{eqnarray*}
Using $\nabla^{c}_{X^{v}}Y^{v} = 0$ (\cite{matsuzoe2003statistical}) we have 
\begin{eqnarray}\label{cs1}
\nabla^{c}_{X^{H}}Y^{H} &=& (\nabla_{X}Y)^{c} - [\nabla^{c}_{\gamma (\nabla X)}Y^{c} + \nabla^{c}_{X^{c}}\gamma(\nabla Y)]
\end{eqnarray}
By definition 
\begin{eqnarray}\label{cs2}
(\nabla_{X}Y)^{c} = (\nabla_{X}Y)^{H}+ \gamma(\nabla(\nabla_{X}Y))
\end{eqnarray}
From $(\ref{cs1})$ and $(\ref{cs2})$  
\begin{eqnarray*}
\mathcal{H}(\nabla^{c}_{X^{H}}Y^{H}) = (\nabla_{X}Y)^{H}.
\end{eqnarray*}
Hence the submersion $\pi : (T\mathbf{M}, \nabla^{c}) \longrightarrow (\mathbf{M},\nabla)$ is an affine submersion with horizontal distribution.
\end{proof}

\begin{proposition}\label{csp2}
The submersion $\pi: (T\mathbf{M}, g^{s})\longrightarrow (\mathbf{M},g)$ is a semi-Riemannian submersion.
\end{proposition}

\begin{proof}
Clearly $\pi^{-1}(p) = T_{p}\mathbf{M}$ for all $p\in \mathbf{M}$ is a semi-Riemannian submanifold of $T\mathbf{M}$ and by definition of $g^{s}$ we have 
\begin{eqnarray*}
g^{s}(X^{H},Y^{H}) = g(X,Y).
\end{eqnarray*}
Hence $\pi$ is a semi-Riemannian submersion.
\end{proof}

Now we give a necessary and sufficient condition for the tangent bundle to be a statistical manifold with the Sasaki lift metric and the complete lift connection.

\begin{theorem}
$(T\mathbf{M},\nabla^{c},g^{s})$ is a statistical manifold if and only if 
\begin{enumerate}
\item $\mathcal{H}(S_{V}X) = A_{X}V - \overline{A}_{X}V$
\item $\mathcal{V}(S_{X}V) = T_{V}X - \overline{T}_{V}X$
\item $(T_{p}\mathbf{M},\hat{\nabla}^{c},\hat{g}^{s})$ is a statistical manifold for each $p\in \mathbf{M}$.
\item $(\mathbf{M},\nabla,g)$ is a statistical manifold.
\end{enumerate}
\end{theorem}


\begin{proof}
From propositions (\ref{csp1}) and (\ref{csp2}) we get that $\pi: (T\mathbf{M},\nabla^{c}, g^{s})\longrightarrow (\mathbf{M}, \nabla,g)$ is an affine submersion with horizontal distribution. Since $g^{s}(X^{H},Y^{V}) = 0$, we can take $\mathcal{H}_{\nabla}(\mathbf{M}) = \mathcal{V}(\mathbf{M})^{\perp}$. First we show that the following equations holds for horizontal vectors $X,Y$ and vertical vectors $U,V,W$
\begin{eqnarray}
\label{cst1}(\nabla^{c}_{V}g^{s})(X,Y) &=& -g^{s}(S_{V}X,Y)\\
\label{cst2}(\nabla^{c}_{X}g^{s})(V,Y) &=& -g^{s}(A_{X}V,Y)+ g^{s}(\overline{A}_{X}V,Y)\\
\label{cst3}(\nabla^{c}_{X}g^{s})(V,W) &=& -g^{s}(S_{X}V,W)\\
\label{cst4}(\nabla^{c}_{V}g^{s})(X,W) &=& -g^{s}(T_{V}X,W)+ g^{s}(\overline{T}_{V}X,W)\\
\label{cst5}(\nabla^{c}_{U}g^{s})(V,W) &=& (\hat{\nabla}^{c}_{U}\hat{g}^{s})(V,W)\\
\label{cst6}(\nabla^{c}_{\tilde{X}} g^{s})(\tilde{X}_{1},\tilde{X}_{2}) &=& (\nabla_{X}g)(X_{1},X_{2}) 
\end{eqnarray}
where $\tilde{X}_{i}$ are the horizontal lift of vector fields $X_{i}$ on $\mathbf{M}$, $\hat{g}^{s}$ is the induced metric on the fibers and $S_{V}X = \nabla^{c}_{V}X - \overline{\nabla}^{c}_{V}X$. To see (\ref{cst1}) consider
\begin{eqnarray*}
(\nabla^{c}_{V}g^{s})(X,Y) &=& Vg^{s}(X,Y)-g^{s}(\nabla^{c}_{V}X, Y)-g^{s}(X,\nabla^{c}_{V}Y)\\
&=& g^{s}(\overline{\nabla}^{c}_{V}X,Y)- g^{s}(X,\nabla^{c}_{V}Y)\\
&=&  -g^{s}(S_{V}X,Y) 
\end{eqnarray*}
Similarly we can prove the other equations.Now suppose $(T\mathbf{M},\nabla^{c},g^{s})$is a statistical manifold, then $\nabla^{c} g^{s}$ is symmetric. From (\ref{cst1}) and (\ref{cst2}) we get 
\begin{eqnarray*}
\mathcal{H}(S_{V}X) = A_{X}V - \overline{A}_{X}V
\end{eqnarray*}
From (\ref{cst3}) and (\ref{cst4}) we get 
\begin{eqnarray*}
\mathcal{V}(S_{X}V) = T_{V}X - \overline{T}_{V}X
\end{eqnarray*}
from (\ref{cst5}) $\hat{\nabla}^{c}\hat{g}^{s}$ is symmetric, so  $(T_{p}\mathbf{M},\hat{\nabla}^{c},\hat{g}^{s})$ is a statistical manifold for each $p\in \mathbf{M}$.  
Also from (\ref{cst6})  $(\mathbf{M},\nabla,g)$ is a statistical manifold.\\ Conversely, if all the four conditions hold then from the above equations $\nabla^{c}g^{s}$ is symmetric, so $(T\mathbf{M},\nabla^{c},g^{s})$  is a statistical manifold.
\end{proof}

\section*{Acknowledgements}
The first named author was supported by the Doctoral Research Fellowship from
the Indian Institute of Space Science and Technology (IIST), Department of Space, Government of India.

\bibliographystyle{plain}

\begin{thebibliography}{10}

\bibitem{abe2001affine}
N~Abe and K~Hasegawa.
\newblock An affine submersion with horizontal distribution and its
  applications.
\newblock {\em Differential Geometry and its Applications}, 14(3):235--250,
  2001.

\bibitem{amari2016information}
S~Amari.
\newblock {\em Information geometry and its applications}, volume 194.
\newblock Springer, 2016.

\bibitem{amari2007methods}
S~Amari and H~Nagaoka.
\newblock {\em Methods of information geometry}, volume 191.
\newblock American Mathematical Soc., 2007.

\bibitem{barndorff1988differential}
OE~Barndorff-Nielsen and PE~Jupp.
\newblock Differential geometry, profile likelihood, {L}-sufficiency and
  composite transformation models.
\newblock {\em The Annals of Statistics}, pages 1009--1043, 1988.

\bibitem{gud1990har}
S~Gudmundsson.
\newblock On the geometry of harmonic morphisms.
\newblock {\em Mathematical Proceedings of the Cambridge Philosophical
  Society}, 108:461--466, 1990.

\bibitem{kurose1990dual}
T~Kurose.
\newblock Dual connections and affine geometry.
\newblock {\em Mathematische Zeitschrift}, 203(1):115--121, 1990.

\bibitem{lauritzen1987statistical}
S~L Lauritzen.
\newblock Statistical manifolds.
\newblock {\em Differential geometry in statistical inference}, 10:163--216,
  1987.

\bibitem{matsuzoe2003statistical}
H~Matsuzoe and J~Inoguchi.
\newblock Statistical structures on tangent bundles.
\newblock {\em Appl. Sci}, 5:55--75, 2003.

\bibitem{nomizu1994affine}
K~Nomizu and T~Sasaki.
\newblock {\em Affine differential geometry: geometry of affine immersions}.
\newblock Cambridge university press, 1994.

\bibitem{n1983semi}
B~O'~Neill.
\newblock {\em Semi-Riemannian Geometry with Application to Relativity}.
\newblock Academic Press, 1983.

\bibitem{o1966fundamental}
B~O'Neill et~al.
\newblock The fundamental equations of a submersion.
\newblock {\em The Michigan Mathematical Journal}, 13(4):459--469, 1966.

\bibitem{o1967submersions}
Barrett O{'}Neill.
\newblock Submersions and geodesics.
\newblock {\em Duke Mathematical Journal}, 34(2):363--373, 1967.

\bibitem{lornia1993fe}
L~Ornea and G~Romani.
\newblock The fundamental equations of conformal submersions.
\newblock {\em Beitrage Algebra Geom}, 34(2):233--243, 1993.

\bibitem{yano1973tangent}
K~Yano and S~Ishihara.
\newblock {\em Tangent and cotangent bundles: Differential geometry},
  volume~16.
\newblock Dekker, 1973.

\end{thebibliography}
 \nocite{o1966fundamental}

\end{document}